\documentclass{amsart}

\usepackage{graphicx}
\usepackage{amssymb}
\usepackage{amsthm}
\usepackage{amsmath}
\usepackage{pstricks,pst-node}
\input xy
\xyoption{all}
\SelectTips{cm}{}
\newtheorem{theorem}{Theorem}
\newtheorem{lemma}[theorem]{Lemma}

\newtheorem{corollary}[theorem]{Corollary}

\theoremstyle{definition}
\newtheorem{definition}[theorem]{Definition}

\newcommand{\Z}{\mathbb{Z}}
\newcommand{\Q}{\mathbb{Q}}
\newcommand{\F}{\mathbb{F}}

\newcommand{\inv}{^{-1}}

  \def \bP {{\bf P}}

             \def \cH {{\mathcal H}}

             \def \cS {{\mathcal S}}
 \def \cC {{\mathcal C}}
\def \cP {{\mathcal P}}
\def \cV {{\mathcal V}}
             \def \cW {{\mathcal W}}

\def\F{\mathbb F}

\def\R{\mathbb R}

\def\Q{\mathbb Q}
\def\Z{\mathbb Z}

\def\inv{^{-1}}

\begin{document}

\title[Resolutions of the Steinberg module]{Resolutions of the Steinberg module for $GL(n)$}

\author{Avner Ash} \address{Boston College\\ Chestnut Hill, MA 02445}
\email{Avner.Ash@bc.edu} \author{ Paul E. Gunnells}
\address{University of Massachusetts Amherst\\ Amherst, MA 01003}
\email{gunnells@math.umass.edu} \author{Mark McConnell}
\address{Center for Communications Research\\ Princeton, New Jersey 08540}
\email{mwmccon@idaccr.org}

\thanks{
AA wishes to thank the National Science Foundation for support of this
research through NSF grant DMS-0455240, and also the NSA through grant
H98230-09-1-0050.  This manuscript is submitted for publication with
the understanding that the United States government is authorized to
produce and distribute reprints.  PG wishes to thank the National
Science Foundation for support of this research through NSF grant
DMS-0801214.}



\date{June 24, 2011}

\maketitle

\begin{abstract}
We give several resolutions of the Steinberg representation $St_{n}$
for the general linear group over a principal ideal domain, in
particular over $\Z$.  We compare them, and use these results to prove
that the computations in~\cite{AGM4} are definitive.  In particular,
in \cite{AGM4} we use two complexes to compute certain cohomology
groups of congruence subgroups of $SL(4,\Z)$.  One complex is based on
Voronoi's polyhedral decomposition of the symmetric space for $SL
(n,\R)$, whereas the other is a larger complex that has an action of
the Hecke operators.  We prove that both complexes allow us to compute
the relevant cohomology groups, and that the use of the Voronoi
complex does not introduce any spurious Hecke eigenclasses.
\end{abstract}

\section{Introduction}\label{intro}

In a series of papers \cite{AGM1, AGM2, AGM3, AGM4} we computed the
cohomology $H^{5}$ with constant coefficients of certain
congruence subgroups $\Gamma\subset SL(4,\Z)$, with $5$ being chosen
since this is the topmost degree that can contain classes corresponding to
cuspidal automorphic forms \cite[\S 1]{AGM1}.  We computed the action
of the Hecke operators on the cohomology and studied connections with
representations of the absolute Galois group of $\Q$.

The main tool in our computations is the \emph{Steinberg module}
$St_n$, which is the dualizing module of any torsionfree congruence
subgroup of $SL(n,\Z)$ \cite{B-S}. We used a variant of a resolution
of the Steinberg module for $GL(4,\Z)$ first published in \cite{LS}.
This variant is called the \emph{sharbly complex}.  In our previous
papers, we only asserted that the sharbly complex is a resolution of
$St_n$; in this paper, we provide a proof.

When multiplication by $30$ is not invertible in the coefficient
module, our computations were not definitive for the following
reason.  To compute the cohomology we use the \emph{Voronoi complex},
which is a chain complex built from an $SL (n,\Z)$-equivariant
polytopal tessellation of the symmetric space $SL(n,\R)/SO (n)$.  This
tessellation is finite modulo $SL(n,\Z)$, and thus provides a
convenient tool for explicit computations.  But to compute the action
of the Hecke operators on cohomology, we must use the sharbly complex
instead of the Voronoi, since the Hecke operators act on the former
and not on the latter.  There is a map from the Voronoi-based
cohomology in degree 5 to the sharbly-based cohomology in that degree
that we thought we had to assume was injective for our results to be
meaningful.  As far as we know, this map may fail to be injective.
However, in this paper, we prove that if $n = 4$, the Hecke
eigenvalues we compute on $H^5$ are meaningful. We show that they are
Hecke eigenvalues in the homology of $\Gamma$ with coefficients in the
sharbly complex.  We prove a similar result for $n=3$.

Given any congruence subgroup $\Gamma \subset SL (n,\Z)$ and any $\Z
\Gamma$-module $M$, one can compute the homology of the sharbly
complex with coefficients in $M$.  If $M$ is a vector space over
$\F_p$ for a prime $p$ that is not a torsion prime of $\Gamma$, or
over a field of characteristic zero, then the sharbly homology is
isomorphic to the group cohomology of $\Gamma$.  For torsion primes,
however, what one computes with the sharbly complex is not so clear.
Our main result in this direction (Corollary \ref{S-Steinberg}) is
that if $p$ is any odd prime, then the sharbly homology is isomorphic
to the \emph{Steinberg homology} with coefficients in $M$; by
definition this is $H_{*} (\Gamma ,St_{n}\otimes_{\Z}M)$.  Note that
these latter groups are exactly the group cohomology of $\Gamma$ with
coefficients in $M$ if $\Gamma $ is torsionfree.  Furthermore, we
prove in Corollary~\ref{V-Steinberg} that if $n\le 4$ the Voronoi
homology (Definition \ref{def:voronoi.homology}) is isomorphic to the
Steinberg homology.  This was left open in~\cite{AGM4}.

We remark that the key to making our results work in odd
characteristics is to replace alternating chains with ordered chains
in the sharbly complex.  Our result on the meaningfulness of our Hecke
eigenvalue computations depends on using Lee and Szczarba's original
resolution of the Steinberg module.

\section{The resolution of Lee and Szczarba of the Steinberg
module}\label{ls}

In this section we briefly review the relevant contents of \cite[\S
3]{LS}.  We phrase Lee and Szczarba's construction slightly
differently, but equivalently, in a way that is more suitable for our
purposes.

Let $n\ge1$.  For any ring $A$, let $A^n$ denote the right
$GL(n,A)$-module of row vectors.  We say that $v\in A^n$ is
\emph{unimodular} if the entries of $v$ generate the unit ideal in
$a$.  Let $M_{s,t}(A)$ denote the set of all $s\times t$ matrices with
entries in $A$.  For any set $U$, let $\Z(U)$ denote the free abelian
group generated by $U$.  All homology will be taken with $\Z$
coefficients unless otherwise indicated.

Now we assume $A$ is a principal ideal domain.  Let $F$ be its field
of fractions, and assume $n\ge 2$.  The \emph{Tits building} $T_n(A)$
is the simplicial complex whose vertices are the proper nonzero
subspaces of $F^n$ and whose simplices correspond to flags of
subspaces.  Note that $T_n(A)$ depends only on $F$.  By the
Solomon--Tits theorem $T_{n} (A)$ has the homotopy type of a wedge of
$(n-2)$-dimensional spheres.  It is a right $GL(n,F)$-module and
therefore so is its homology.  We define the \emph{Steinberg module}
to be the reduced homology of the Tits building:
\begin{equation}\label{eq:steinberg-def}
St_n(A)=\widetilde H_{n-2}(T_n(A)).
\end{equation}

For each hyperplane $H$ of $F^n$ let $S(H)$ be the subcomplex of
$T_n(A)$ consisting of all simplices of $T_n(A)$ with $H$ as a vertex.
The $S(H)$ form an acyclic cover of $T_n(A)$, and for all $q\ge0$ the
nerve $\widetilde N$ of this cover satisfies
$$
H_{q}(T_n(A))=H_{q}(\widetilde N).
$$

For $k\ge0$ let $\cS_k \subset M_{n+k,n}(A)$ be the set of all
matrices $M$ with all rows unimodular.  Let $\cP_k \subset \cS_{k}$ be
the subset of matrices with rank $<n$.  These sets have a natural
right action by $GL(n,A)$.  We define the boundary operator
$$
\partial\colon\Z(\cS_k)\to \Z(\cS_{k-1})
$$
by $\partial M = \sum_{i=1}^{n+k} (-1)^i M_i$, where $M_i$ is the
matrix formed from $M$ by deleting its $i$th row.  Since $\partial$
takes $\Z(\cP_k)$ to $\Z(\cP_{k-1})$, we can form the quotient complex
$$
\cC_k(A):=\Z(\cS_k)/\Z(\cP_k).
$$
This is a complex of right $GL(n,A)$-modules.

Theorem 3.1 of \cite{LS} asserts that there is an epimorphism $\phi\colon
\cC_0(A)\to St_n(A)$ such that
$$
\cdots\to \cC_k(A) \to \cC_{k-1}(A) \to \cdots \to \cC_0(A) \to St_n(A) \to 0
$$
is a free $GL(n,A)$-resolution of $St_n(A)$.
For the convenience of the reader, and since we will need to use
similar arguments in the sequel, we sketch the proof.

Let $K$ be the simplicial complex whose vertices are the unimodular
elements in $A^n$ and whose simplices are all finite nonempty subsets
of vertices.  Let $L$ be the subcomplex of $K$ consisting of those
simplices all of whose vertices lie in one and the same proper direct
summand of $A^n$.  The group $GL(n,A)$ acts on the right of $K$ and
$L$.  Since $K$ is acyclic, the exact sequence of the pair $(K,L)$
implies $H_k(K,L) = \widetilde H_{k-1}(L)$
for all $k\ge 0$.

If $M$ is a simplicial complex or a pair of such, let $C_{*}(M)$
denote the \emph{ordered} simplicial chain complex on $M$
\cite[\S4.3]{spanier}.  The following is Lemma~3.2 of \cite{LS}:

\begin{lemma}\label{key}
$H_{q}(K,L)=0$ if $q\ne n-1$ and $H_{n-1}(K,L)\approx St_n(A)$.
\end{lemma}

Assume the lemma.  The $(n-2)$-skeletons of $L$ and $K$ are the same, so
$C_{n+k-1}(K,L)=0$ if $k<0$.  We obtain an exact sequence
$$
\cdots\to C_{n+k}(K,L)\to C_{n+k-1}(K,L) \to\cdots\to C_{n-1}(K,L) \to
St_n(A)\to 0.
$$

We can map a simplex $\sigma$ in $K$ to the matrix whose rows consist
of the vertices of $\sigma$, in order.  This gives us isomorphisms for
$k\ge0$:
$$
C_{n+k-1}(K)\approx\Z(\cS_k) \ \  \mbox{and} \  \ C_{n+k-1}(L)\approx\Z(\cP_k).
$$ 
Therefore $C_{n+k-1}(K,L)=\cC_k(A)$ and we have an exact sequence
$$
\cdots\to \cC_{k+1}(A)\to \cC_{k}(A) \to\cdots\to \cC_{0}(A) \to St_n(A)\to 0.
$$

It remains to prove Lemma~\ref{key}.

\begin{proof}
Let $\cH$ be the set of direct summands of rank $n-1$ in $A^n$.  Since
$A$ is a PID, $H$ is isomorphic to $A^{n-1}$.  For $H\in\cH$, let
$K_H$ denote the subcomplex of $L$ consisting of all simplices whose
vertices lie in $H$.  For the same reason that $K$ is contractible, so
is $K_H$.  More generally, if $H_1,\dots H_q\in\cH$, then
$H_1\cap\cdots\cap H_q$ is isomorphic to $A^{n-q}$ and
$K_{H_1}\cap\cdots\cap K_{H_q}$ is contractible.

Therefore, $\{K_H\}$ is an acyclic cover of $L$.  Letting $N$ denote
its nerve, we have for all $q\ge0$
$$
H_q(L)\approx H_q(N).
$$
The map $H\mapsto H\otimes_A F$ defines a simplicial isomorphism
$N\to\widetilde N$.  We obtain a sequence of $GL(n,A)$-equivariant
isomorphisms:
$$
H_q(K,L)\approx \widetilde H_{q-1}(L)\approx
\widetilde H_{q-1}(N)\approx\widetilde H_{q-1}(\widetilde N)
\approx\widetilde H_{q-1}(T_n(A)).
$$
By the Solomon--Tits theorem and \eqref{eq:steinberg-def}, the proof
is complete.
\end{proof}

It is easy to see that $\cC_{*}$ is free as a $GL(n,A)$ module.

In \cite{AR}, to each matrix $X$ in $GL(n,F)$ is associated the
modular symbol $[[X]]\in St_n(A)$, which is the fundamental class of
the apartment corresponding to $X$ in the Tits building.  The map
$\phi\colon\cC_{0}(A) \to St_n(A)$ may be taken to be $M\mapsto[[M]]$, and
we will do so in the sequel.

\section{A variant resolution of the Steinberg module}\label{vls}

In this section we present a variant of the construction of Lee and
Szczarba and prove that it gives a resolution of the Steinberg module
as well.  We keep the same notation as in the preceding section.

Let $\bP^{n-1}(F)$ be the set of lines in $F^n$.  Set $\cS'_k =
(\bP^{n-1}(F))^{n+k}$ and let $\cP'_k$ be the subset of $\cS'_k$ where
the lines in the $(n+k)$-tuple do not span $F^n$.  The right
$GL(n,F)$-action on $F^n$ induces an action on $\cS'_k$ that preserves
$\cP'_k$.

There is an obvious $GL(n,A)$-equivariant quotient map
$\theta_k\colon\cS_k \to \cS'_k$ that induces a map $\cP_k\to
\cP'_k$.  The boundary operator $\partial$ induces a boundary operator
that we also denote by $\partial$:
$$
\partial\colon\Z(\cS'_k)\to \Z(\cS'_{k-1}).
$$
As before $\partial$ takes $\Z(\cP'_k)$ to $\Z(\cP'_{k-1})$, and we
can form the
quotient $GL(n,F)$-complex
$$
\cC'_k(A):=\Z(\cS'_k)/\Z(\cP'_k).
$$
Note that $\phi\colon \cC_{0}(A) \to St_n(A)$ is constant on the
fibers of $\theta_0$ and therefore induces an epimorphism $\phi'\colon
\cC'_0(A)\to St_n(A)$.

\begin{theorem}\label{C'} The exact sequence 
\begin{equation}\label{eqn:resolution}
\cdots\to \cC'_k(A) \to \cC'_{k-1}(A) \to \cdots \to \cC'_0(A) 
\stackrel{\phi'}{\to} St_n(A) \to 0
\end{equation}
is a $GL(n,F)$-resolution of $St_n(A)$.  
\end{theorem}

We remark that \eqref{eqn:resolution} is not a free
$GL(n,A)$-resolution if $A$ has more than one unit.

\begin{proof}
Let $K'$ be the simplicial complex whose vertices are the elements of
$\bP^{n-1}(F)$ and whose simplices are all finite nonempty subsets of
vertices.  Let $L'$ be the subcomplex of $K'$ consisting of those
simplices all of whose vertices lie in one and the same proper direct
summand of $F^n$.  The group $GL(n,F)$ acts on the right of $K'$ and
$L'$.  Since $K'$ is acyclic, we have $H_k(K',L') = \widetilde
H_{k-1}(L')$ for all $k\ge 0$ from the exact sequence of the pair
$(K',L')$.

\begin{lemma}\label{key'}
$H_{q}(K',L')=0$ if $q\ne n-1$ and $H_{n-1}(K',L')\approx St_n(A)$ via $\phi'$.
\end{lemma}

\begin{proof}[Proof of Lemma \ref{key'}]

Let $\cH'$ be the set of direct summands of rank $(n-1)$ in $F^n$.
Since $F$ is a field, any $H\in \cH'$ is isomorphic to $F^{n-1}$.  For $H\in\cH'$,
let $K'_H$ denote the subcomplex of $L'$ consisting of all simplices
whose vertices lie in $H'$.  For the same reason that $K'$ is
contractible, so is $K'_H$.  More generally, if $H_1,\dots H_q\in\cH$,
then $H_1\cap\cdots\cap H_q$ is isomorphic to $F^{n-q}$, and
$K'_{H_1}\cap\cdots\cap K'_{H_q}$ is contractible.

Therefore $\{K'_H\}$ is an acyclic cover of $L'$.  Letting $N'$
denote its nerve, we have for all $q\ge0$
$$
H_q(L')\approx H_q(N').
$$
The identity map $H\mapsto H$ defines a simplicial isomorphism
$N'\to\widetilde N$.  We obtain a sequence of $GL(n,F)$-equivariant
isomorphisms
$$
H_q(K',L')\approx \widetilde H_{q-1}(L')\approx
\widetilde H_{q-1}(N')\approx\widetilde H_{q-1}(\widetilde N)
\approx\widetilde H_{q-1}(T_n(A)).
$$

Let $X\in GL(n,F)$ be associated to the modular symbol
$[[X]]\in St_n(A)$\label{modsymbdef}.  The apartment corresponding to $X$ in the Tits
building depends only on the lines in $F^n$ generated by the rows of
$X$.  Thus the map $\phi'$ behaves as claimed, and this proves the
lemma.
\end{proof}

We return now to the proof of Theorem \ref{C'}.  The $(n-2)$-skeletons
of $L'$ and $K'$ are the same, so $C_{n+k-1}(K',L')=0$ if $k<0$.  We
obtain an exact sequence
$$
\cdots\to C_{n+k}(K',L')\to C_{n+k-1}(K',L') \to\cdots\to C_{n-1}(K',L') \to St_n(A)\to 0.
$$

Clearly, for $k\geq 0$ we have isomorphisms
$$
C_{n+k-1}(K')\approx\Z(\cS'_k) \ \  \mbox{and} \  \ C_{n+k-1}(L')\approx\Z(\cP'_k).
$$ 
Therefore $C_{n+k-1}(K',L')=\cC'_k(A)$, and we have an exact sequence
$$
\cdots\to \cC'_{k+1}(A)\to \cC'_{k}(A) \to\cdots\to \cC'_{0}(A) 
\stackrel{\phi'}{\to} St_n(A)\to 0.
$$
This completes the proof of the theorem.
\end{proof}

\section{The sharbly complex}\label{sharblies}

For the purposes of computing Hecke operators, one needs to modify
the construction of Lee and Szczarba to obtain a resolution of
$St_n(A)$ by $GL(n,F)$-modules.  One possibility is to take $A=F$ obtaining a free
$GL(n,F)$-complex $\cC(F)$, thereby allowing in principle an easy
formula for the action of Hecke operators on the homology.  However,
this resolution is far too big to be used for practical computation.

To make Lee and Szczarba's complex ``smaller," we simultaneously
antisymmetrize and factor out the action of scalars on row vectors,
even though in this way we sacrifice freeness.  We obtain the sharbly
complex $Sh_{*}$, which gives a resolution of the Steinberg
module by $GL(n,F)$-modules. In section \ref{voronoi}, we will
introduce much smaller resolutions using Voronoi theory.

Let $\Gamma$ be a subgroup of finite index in $GL(n,A)$.  The sharbly
resolution is $\Gamma$-free if $\Gamma$ is torsionfree, but in general
it is not even $\Gamma$-projective.

The sharbly complex was defined in \cite{A2}, in a slightly different
form from the definition in \cite{AGM4}.  It is straightforward to
see that the two different definitions give naturally isomorphic
complexes of $GL(n,\Z)$-modules.  The advantage of the latter
definition is that there is an obvious $GL(n,\Q)$-action on the
complex.  We give here the latter form, generalized from $\Z$ to an
arbitrary principal ideal domain.

We continue to use the notation of the preceding sections.

\begin{definition}
The \emph{Sharbly complex} $Sh_{*} = Sh_{*} (A)$ is the
complex of $\Z GL(n,F)$-modules defined as follows.   As an abelian
group, $Sh_{k} (A)$ is generated by 
symbols $[v_1,\dots,v_{n+k}]$, where the $v_i$ are nonzero vectors in
$F^n$, modulo the submodule generated by the following relations:

(i) $[v_{\sigma
(1)},\dots,v_{\sigma(n+k)}]-(-1)^\sigma[v_1,\dots,v_{n+k}]$ for all
permutations $\sigma$;

(ii) $[v_1,\dots,v_{n+k}]$ if $v_1,\dots,v_{n+k}$ do not span all
of $F^n$; and

(iii) $[v_1,\dots,v_{n+k}]-[av_1,v_{2},\dots,v_{n+k}]$ for all $a\in
F^\times$.

\noindent The boundary map $\partial \colon Sh_{k} (A) \rightarrow
Sh_{k-1} (A)$ is given by 
\[ 
\partial([v_1,\dots,v_{n+k}])=
\sum_{i=1}^{n+k}[v_1,\dots,\widehat{v_i},\dots v_{n+k}],
\]
where as usual $\widehat{v_i}$ means to delete $v_{i}$.
\end{definition}

Writing $v_1,\dots,v_n$ as row vectors, we let $X(v_1,\dots,v_n)$
denote the matrix with the $v_{i}$ as rows, and put
$[[v_1,\dots,v_{n}]]=[[X(v_1,\dots,v_{n})]]$ as in the proof of Lemma
\ref{key'}.  The map $[v_1,\dots,v_{n}] \mapsto [[v_1,\dots,v_{n}]]$
is constant on the cosets of the group generated by the relations
(i)--(iii) and thus defines a surjective $GL(n,F)$-equivariant map
$\phi_{Sh}\colon Sh_0(A) \to St_n(A)$.

For each line $\ell_i\in\bP^{n-1}(F)$ choose a nonzero unimodular
vector $v_i\in\ell_i\cap A^n$.  The map
$(\ell_1,\dots,\ell_{n+k})\mapsto [v_1,\dots,v_{n+k}]$ extends to a
$GL(n,F)$-equivariant chain map $f\colon\cC'_{*}(A)\to Sh_{*}(A)$ and
$\phi'=\phi_{Sh}\circ f$.

\begin{theorem}\label{Sh}  The following is an exact sequence of $GL(n,F)$-modules:
$$
\cdots\to Sh_k(A) \to Sh_{k-1}(A) \to \cdots \to Sh_0(A) 
\stackrel{\phi_{Sh}}{\to} St_n(A) \to 0.
$$
\end{theorem}

\begin{proof} 
We interpret $f$ geometrically as follows.  Since the
$(n-2)$-skeletons of $K'$ and $L'$ are the same, $\cC'_{*}(A)$ is
naturally isomorphic to the complex of ordered chains of the pair of
simplicial complexes $(K',L')$ and the sharbly complex $Sh_{*}(A)$ is
naturally isomorphic to the complex of oriented (antisymmetric) chains
of $(K',L')$.  A standard fact tells us that the complex of ordered
chains on a simplicial complex $Y$ is homotopy equivalent to the
complex of oriented chains on $Y$ (cf.~\cite[Theorem 4.3.8]{spanier}).
We take the long exact sequence in ordered homology coming from the
pair $(K',L')$ with its natural map to the long exact sequence in
oriented homology coming from $(K',L')$.  Using Theorem \ref{C'},
applying the ``standard fact'' to the homology of $K'$ and $L'$
respectively, plus repeated use of the Five Lemma gives the result.
\end{proof}

\begin{definition} 
Let $\Gamma$ be a subgroup of $GL(n,F)$ and $M$ a right
$\Z\Gamma$-module.  Consider $M$ to be a complex concentrated in
dimension 0.  We define the \emph{Steinberg homology} of $\Gamma$ with
coefficients in $M$ to be $H_*(\Gamma,St_{n}(A) \otimes_\Z M)$.  We
define the \emph{sharbly homology} of $\Gamma$ with coefficients in
$M$ to be $H_*(\Gamma,Sh_{*} \otimes_\Z M)$.
\end{definition}

Note that if $\Gamma$ is torsionfree, then Borel--Serre duality
\cite{B-S} implies that the Steinberg homology is isomorphic to the
group cohomology of $\Gamma$.  Also, if $P$ is a projective resolution
of $\Z$ over $\Z\Gamma$, then by definition
$$
H_*(\Gamma,Sh_{*}(A) \otimes_\Z M)= H_*(P\otimes_\Gamma (Sh_{*}(A) \otimes_\Z M))
$$
(well-defined up to canonical isomorphism.)  From Theorems~\ref{C'}
and~\ref{Sh} and from the standard spectral sequences of a double
complex (e.g.~\cite[p.~169]{B}), we obtain spectral sequences
$$
E^1_{p,q}=H_q(\Gamma,C_p\otimes_\Z M) \Rightarrow
H_*(\Gamma,St_n(A)\otimes_\Z M),
$$
where $C_{*}=\cC'_{*}(A)$ or $Sh_{*}(A)$.

\begin{theorem}\label{C-Steinberg}
Assume that $A^\times$ is finite of order $o$ and let $d$ be the
product of all $m$ such that $\Gamma$ has a subgroup of order
$m$.  Assume that $d$ is finite.  Suppose that multiplication by the
greatest common divisor $\gcd(o,d)$ is invertible on $M$.  Then we
have isomorphisms
$$
H_*(\Gamma,C_{*}\otimes_\Z M) 
\approx H_*(\Gamma,St_n(A)\otimes_\Z M)
$$
where $C_{*}=\cC'_{*}(A)$ or $Sh_{*}(A)$.
\end{theorem}

\begin{proof}
The chain map $\cC'_{*}(A)\stackrel{f}{\to}Sh_{*}(A)$ is a weak equivalence.
By \cite[Proposition VII.5.2]{B}, the map $f$ induces an isomorphism
on the homology of $\Gamma$ with those coefficients.  So it suffices
to prove the theorem for $C_{*}=\cC'_{*}(A)$.

Use the notation of the proof of Theorem~\ref{Sh}.  A basis for
$\cC'_k(A)$ is given by tuples $B=(\ell_1,\dots,\ell_{n+k})$ such that
the span of the lines $\ell_1,\dots,\ell_{n+k}$ is all of $F^n$.  Let
$Stab_B$ be the stabilizer in $\Gamma$ of $B$ and let $s_B$ be its
order.  Naturally, $s_B$ divides $d$.

We claim that $s_B$ also divides $o^{N}$ for some sufficiently large
integer $N$.  Let $\gamma\in Stab_B$.  Then $\ell_i\gamma=\ell_i$ for
all $i$.  Therefore, for each $i$ there exists $a_i\in A^\times$ such
that $v_i\gamma=a_iv_i$.  The $v_i$ span $F^n$.  It follows that
$\gamma^o=1$.  So any prime dividing the order of any element of
$Stab_B$ divides $o$ and the claim follows.

The $\Gamma$-module $\cC'_k(A)$ is a direct sum of induced
representations, each induced from $Stab_B$ for some $B$.  By
Shapiro's lemma, $H_q(\Gamma,\cC'_p(A)\otimes M)$ is a direct sum of
groups equal to $H_q(Stab_B, M)$ for various $B$.  Since $s_B$ is
invertible on $M$, these groups vanish when $q>0$.  Thus the terms in
the spectral sequence are 0 except when $q=0$, which implies that the
spectral sequence degenerates at $E_2$.  This completes the proof.
\end{proof}

Since $\Z^{\times} = \{\pm 1 \}$, Theorem \ref{C-Steinberg} yields the
following:

\begin{corollary}\label{S-Steinberg}
Let $A=\Z$.  For any $\Gamma\subset GL(n,\Z)$ and any coefficient
module $M$ in which $2$ is invertible, the sharbly homology is
isomorphic to the Steinberg homology.
\end{corollary}

It follows from Corollary~\ref{S-Steinberg} that (a) and (d) of
Conjecture 5 in \cite{AGM4} are equivalent, if $p \ne 2$.

\section{The Voronoi complex and its variants}\label{voronoi}

From now on, we put $F=\Q$ and $A=\Z$.  Let $X_n^0$ be the space of
positive definite real $n\times n$ symmetric matrices.  It is an open
cone in the vector space $Y_n^0$ of all real $n\times n$ symmetric
matrices.  For each non-zero subspace $W$ of $\R^n$ defined over $\Q$,
set $b(W)$ to be the rational boundary component of $X_n^0$ consisting
of the cone of all positive semi-definite real $n\times n$ symmetric
matrices whose kernel is $W$.  The (minimal) Satake bordification
$(X_n^0)^*$ of $X_n^0$ is the union of $X_n^0$ with all the rational
boundary components.  It is convex and hence contractible.

\begin{lemma}\label{lem:below}
Let $n\ge 2$.  If $k\ne n-1$, $\widetilde H_k(\partial X_n^*)=0$ and 
$\widetilde H_{n-1}(\partial X_n^*) \approx St_{n}(\Z)$.
\end{lemma}

\begin{proof}
Let $W$ be a nonzero, proper subspace of $\Q^n$, thus a vertex of the
Tits building $T_n(\Z)$.  Let $X(W)$ is the subset of $X_n^*$
consisting of all semidefinite symmetric matrices whose kernel
contains $W^\perp$.  Note that $X(W)$ is homeomorphic to $X_{\dim
W}^*$, and hence contractible.  Then $\partial X_n^*$ is covered by
the set of $X(H)$ where $H$ runs over hyperplanes of $\Q^n$.  Also, if
$H_1,\dots,H_j$ are hyperplanes, then $X(H_1)\cap \dots \cap X(H_j)$
is nonempty if and only if $H_1\cap \dots \cap H_j \ne \{0\}$, in
which case $X(H_1)\cap \dots \cap X(H_j)=X_{H_1\cap \dots \cap H_j}$.
Thus, the $X(H)$ form an acyclic covering of $\partial X_n^*$ whose
nerve is $T_n(\Z)$.  The result now follows from the spectral sequence
for a covering, see e.g.~Theorem VII.4.4, p. 168 of~\cite{B}.
\end{proof}

If $n\le 4$, we will produce a resolution of $St_{n}(\Z)$
based on the Voronoi decomposition of $(X_n^0)^*$.  
Expositions of the Voronoi decomposition are given in
\cite[II.6]{AMRT} and \cite[Appendix]{G}.

The positive real numbers $\R_+^\times$ act on $(X_n^0)^*$ by
homotheties.  Set $X_n^*=(X_n^0)^*/\R_+^\times$.

If $v\in \Z^n$ is a unimodular row vector, then ${}^tvv$ is a rank 1
matrix in $(X_n^0)^*$ and thus generates a rational boundary
component.  If $v_1,\dots,v_m$ are $m$ such vectors, we let
$\sigma(v_1,\dots,v_m)$ denote the image of the closed convex conical
hull of ${}^tv_1v_1,\dots,{}^tv_mv_m$ in $X_n^*$.

The Voronoi decomposition of $X_n^*$ is the cellulation of $X_n^*$ by
the cells $\sigma_Q=\sigma(v_1,\dots,v_m)$, where $Q$ runs over all
positive definite real $n\times n$ quadratic forms, and where the
nonzero integral vectors that minimize $Q$ over all integral vectors
are exactly $\pm v_1,\dots,\pm v_m$.  This includes the cells in the
rational boundary components of $X_n^*$.  There is a right action of
$GL(n,\Z)$ on $X_n^*$ induced by the action on $X_n^0$: $x\cdot\gamma
= {}^t\gamma x \gamma$.  The Voronoi decomposition is stable under
this action.

A basic fact is that there are a finite number of Voronoi cells modulo
$SL (n, \Z)$.  We will later need to refer to the representatives of
the $SL(n,\Z)$-orbits of some of the Voronoi cells for $n=3,4$ as
tabulated in \cite{M}.  For $n\le 4$, there is only one Voronoi cell
(modulo $SL(n,\Z)$) that is not a simplex, and that occurs in the top
dimension when $n=4$.

Let $V_n$ denote $X_n^*$ considered as a cell complex, with the
Voronoi cellulation.  Denote by $\Z V_{*}$ the oriented chain complex of
$V_n$.  That is, we fix an orientation on each cell of $V_n$.  Then
$(\Z V)_r$ is the free abelian group generated by the oriented cells of
dimension $r$, and the boundary map from $(\Z V)_r$ to $(\Z V)_{r-1}$
sends an oriented $r$-cell to the linear combination of the
$(r-1)$-cells on its boundary that keeps track of the chosen
orientations.

A Voronoi cell $\sigma(v_1,\dots,v_{k})$ lies in a boundary
component of $X_n^*$ if and only if $v_1,\dots,v_{k}$ do not span
$\Q^n$.  Let $\partial X_n^*$ denote the union of all the boundary
components, and let $\partial V_n$ denote $\partial X_n^*$ considered
as a cell complex, with the Voronoi cellulation.  Denote by
$\Z \partial V_{*}$ the oriented chain complex of $\partial V_n$.

Let $\cV_r= \Z V_r /\Z\partial V_r$. Write $((v_1,\dots,v_{k}))$ for
the generator in $\cV_*$ corresponding to the cell
$\sigma(v_1,\dots,v_k)$.  The boundary maps in all these chain
complexes are induced by the boundary maps on the Voronoi cells.

\begin{definition}\label{def:voronoi.homology}
For any coefficient module $M$ and any subgroup $\Gamma \subset GL
(n,\Z)$, the \emph{Voronoi homology} of $\Gamma$ with coefficients in
$M$ is defined to be the homology $H_*(\Gamma, \cV_*\otimes_\Z M)$.
\end{definition}

Now let $n\le 4$.  Then there is only one $GL(n,\Z)$-orbit of
$(n-1)$-dimensional cells in $V_n$.  (See~\cite{M}.)  They are all of
the form $\sigma(w_1,\dots,w_n)$, where $w_1,\dots,w_n$ is a
$\Z$-basis of $\Z^n$.  We define a linear map
$\phi\colon\cV_{n-1}\to St_n(\Z)$ by sending
\begin{equation}\label{eqn:phidot}
\phi((w_1,\dots,w_n))\mapsto [[w_1,\dots,w_n]].
\end{equation}
This is clearly $GL(n,\Z)$-equivariant.

\begin{theorem}\label{V'} Let $n\le 4$.   Then
\begin{equation}\label{eqn:sequence}
0\to\cV_{n(n+1)/2-1}\to\cdots\to \cV_k \to
\cV_{k-1} \to \cdots \to \cV_{n-1}
\stackrel{\phi}{\to} St_n(\Z) \to 0
\end{equation}
is an exact sequence of $GL(n,\Z)$-modules.
\end{theorem}

\begin{proof}
By Lemma \ref{lem:below} we know that $\widetilde
H_{k}(\partial X_n^*) \approx St_{n} (\Z)$ for $k=n-1$ and vanishes in
all other degrees. Thus, the exact homology sequence of the
pair $(X_n^*,\partial X_n^*)$ and the contractibility of $X_n^*$ imply
that $\widetilde H_k(\cV_{*}) \approx 
\widetilde H_k(\partial
X_n^*)$ for all $k$.  
The $(n-2)$-skeleton of $X_n^*$ lies in $\partial
X_n^*$.  Therefore $\cV_k=0$ for $k\le n-2$.  Thus all of
$\cV_{n-1}$ consists of cycles, and its quotient by the image
of the differential from $\cV_{n}$ is isomorphic to
$\widetilde H_{n-1}(\partial X_n^*) \approx St_{n}(\Z)$.

Let $(e_1,\dots,e_n)$ be the standard basis of $\Q^n$.  Following out
the isomorphisms in the paragraph above, we see that
$((e_1,\dots,e_n))\in \cV_{n-1}$ is sent to the class of the
submanifold of positive semidefinite diagonal matrices in 
$\partial X_n^*$, which goes to the modular symbol $[[I_n]]$ in $St_{n}(\Z)$.
Therefore, using $GL(n,\Z)$-equivariance, the composite is the map
$\phi$ from \eqref{eqn:phidot} as desired.

Since $\widetilde H_k(\partial X_n^*)$ vanishes for $k\ge n$, the rest
of the sequence \eqref{eqn:sequence} is exact.
\end{proof}

\begin{corollary}\label{V-Steinberg}
Let $n\le 4$ and $\Gamma\subset GL(n,\Z)$.  Let $d$ be the greatest
common divisor of the orders of the finite subgroups of $\Gamma$.  For
any coefficient module $M$ on which multiplication by $d$ is
invertible, we have isomorphisms
$$
H_*(\Gamma,\cV_{*}\otimes_\Z M) \approx
H_*(\Gamma,St_n(\Z)\otimes_\Z M).
$$
\end{corollary}

\begin{proof}
The proof runs along the same lines as the proof of
Theorem~\ref{C-Steinberg}.  
\end{proof}

\section{Voronoi sharbly homology classes and Hecke
eigenvalues}\label{AGM}

For the remainder of this article, we take $n=3$ or $4$.  We set $m=3$
if $n=3$ and $m=5$ if $n=4$.  Then every Voronoi cell of dimension
$\le m$ is a simplex.  To simplify notation, we regrade the Voronoi
complex and put $\cW_{k}=\cV_{k+n-1}$.  We also drop the $\Z$ from the
notation for $\cC'$, $Sh$ and $St_{n}$ and omit subscripts for
complexes.

For $0\le k \le m$, define the map of $\Z[GL(n,\Z)]$-modules
$$
\theta_k\colon \cW_{k} \to Sh_{k}
$$
as follows: if $\sigma(v_1,\dots,v_{k+n})$ is a Voronoi cell, it is in
fact a simplex, and we set
$\theta_k((v_1,\dots,v_{k+n}))=[v_1,\dots,v_{k+n}]$.  Note that
$\theta_0$ commutes with the maps to $St_{n}$.

In~\cite{AGM4}, where $n=4$, we wished to compute the homology $H_1(Sh)$
(as a Hecke module) of the complex
$$
Sh_2\otimes_{\Z[\Gamma]} M \to Sh_1\otimes_{\Z[\Gamma]} M \to Sh_0\otimes_{\Z[\Gamma]} M
$$
with $\Gamma$ a congruence subgroup of $SL(4,\Z)$ and $M=\Z$ with
trivial $\Gamma$-action.  Similar computations are anticipated for
$n=3$ in the near future.

What we actually computed in~\cite{AGM4} was the homology $H_1(\cW)$ of
$$
\cW_2\otimes_{\Z[\Gamma]} M \to \cW_1\otimes_{\Z[\Gamma]} M
\to \cW_0\otimes_{\Z[\Gamma]} M.
$$
Let $T$ be a Hecke operator and $\{x_i\}$ a basis of $H_1(\cW)$.
Using the algorithm in \cite{experimental}, we found elements $y_i$ in
$H_1(\cW)$ homologous in $H_1(Sh)$ to $\theta_{1,\ast}(x_i)$.  Then we
found the eigenvalues of the linear map that sends $x_i$ to $y_i$.  We
did this because we do not have a good way to compute the action of
$T$ directly on the Voronoi homology.  The reason is that a direct
computation of $T$ would require acting by integral matrices with
determinant greater than 1, and such matrices do not stabilize the
Voronoi decomposition.  See~\cite{AGM4} for more details.

In \cite[\S 5]{AGM4}, we stated that there would be a problem if
$\theta_{1,\ast}$ is not injective, for then some of the Hecke
``eigenvalues" we computed would be meaningless.  However, this was
not accurate.  As far as we know now, $\theta_{1,\ast}$ could fail to
be injective.  (It would be interesting to decide this point.)
Nevertheless, the Hecke eigenvalues we computed in \cite[\S 5]{AGM4}
are in fact actual eigenvalues in the sharbly homology as defined in
Definition 7.  The rest of this section is devoted to proving this
fact.

We refer to Lemma I.7.4 in \cite{B} as FLHA, or the fundamental lemma
of homological algebra. 

Let $A$ be a ring and $M$ a right $A[\Gamma]$-module.  Let $K$
stand for one of the complexes $\cC$, $\cW$ or $Sh$.  Each of these is
a resolution of $St$.  We give $K\otimes_\Z M$ the diagonal
$\Gamma$-action.  Let $F$ be a resolution of $\Z$ by free
$\Gamma$-modules.  Form the double complex
$$
D(A):=F\otimes_{\Z[\Gamma]} (K\otimes_\Z M).
$$ 
It is an $A$-module through the $A$-action on $M$.  We have
$D_{pq}(K)= F_q\otimes_{\Z[\Gamma]} (K_p\otimes_\Z M)$.  The boundary
maps are denoted as $\partial_1\colon D_{pq}(K)\to D_{p-1,q}(K)$ and
$\partial_2\colon D_{pq}(K)\to D_{p,q-1}(K)$.  The total differential
$\partial$ on $D_{pq} (K)$ is given by $\partial =\partial_1 +(-1)^p
\partial_2$.  Then $H(K):=H(\Gamma,K\otimes_\Z M)$ is the total
homology of $D(K)$.

Referring for example to section 5 of chapter VII of \cite{B}, we have
two spectral sequences that both abut to $H(K)$.  We will refer to the
``first spectral sequence" as the one whose $E^2$ page is
$$
{}_IE^2_{pq}(K)=H_q(\Gamma,H_p(K)) \Rightarrow H_{p+q}(K);
$$
and to the ``second spectral sequence" as the one whose $E^1$ page is
$$
{}_{II}E^1_{pq}(K)=H_q(\Gamma,C_p(K)) \Rightarrow H_{p+q}(K).
$$

Let $S$ be a subsemigroup of $GL(n,\Q)$ such that $(\Gamma,S)$ is a
Hecke pair.  From now on we assume that $M$ is a right $S$-module and
that $F$ is a resolution of $\Z$ consisting of $S$-modules.  Consider
$s\in S$ and $T$ the Hecke operator $\Gamma s\Gamma$. Then the action
of $T$ on $H(K)$ can be computed on the chain level as follows.

Write $\Gamma s\Gamma=\coprod s_i\Gamma$ as a finite disjoint union of
single cosets, with $s_i\in S$.  If $x$ is a cycle in
$(D(K),\partial)$, denote the class in $H(K)$ that it represents by
$[x]$.  Then $[\sum xs_i]=[x]T$.  This gives us a non-canonical
lifting of $T$ to the chain level.  In other words, write ${\bf s} =
\sum s_i$ and view it as an operator on the right on $D(K)$.  Then
${\bf s}$ commutes with the boundary operators and $[x]T=[x\bf s]$ for
any cycle $x$.

Let $\tau_mK$ be the truncation of $K$ at degree $m$, i.e.~the complex
$K_m\to K_{m-1}\to\cdots\to K_0$.  The maps $\theta_k$ above define a
map of complexes $\theta\colon \tau_m\cW\to \tau_mSh$.  Since $\cC$ consists
of free $\Z[\Gamma]$-modules, the FLHA gives us a map of
$\Z[\Gamma]$-complexes $\phi_\cW\colon \cC\to\cW$ that commutes with the
augmentation maps to $St$.  We obtain the map of
$\Z[\Gamma]$-complexes $\theta\circ\phi_\cW\colon \tau_m\cC\to \tau_mSh$.
Using the FLHA, we can extend this to a map of $\Z[\Gamma]$-complexes
$\phi_{Sh}\colon \cC\to Sh$ so that
$(\phi_{Sh})_k=\theta\circ(\phi_\cW)_k$ if $k\le m$.  Again,
$\phi_{Sh}$ commutes with the augmentation maps to $St$.

It follows that $\phi_{\cW}$ and $\phi_{Sh}$ are weak equivalences of
$\Gamma$-chain complexes.  Therefore, by Proposition VII.5.2 in
\cite{B} (which uses the first spectral sequence), they induce
isomorphisms $H(\cC)\stackrel{\sim}{\to} H(\cW)$ and $H(\cC)
\stackrel{\sim}{\to} H(Sh)$ respectively.

\begin{theorem}\label{nofake}
Let $x\in D_{1,0}(\cW)$ be a chain such that $\partial_1(x)=0$ and
$[x]\in {}_{II}E^2_{1,0}(\cW)-\{0\}$.  Let $T$ be a Hecke operator,
$a\in K$ and assume that $[\theta(x)]T=a[\theta(x)]$.  Then there
exists $\xi\in D_{0,1}(Sh)$ such that (i) $\xi+\theta(x)$
is a cycle in $D_1(Sh)$ representing a nonzero class $z\in H_1(Sh)$,
and (ii) $zT=az$.
\end{theorem}

\begin{proof}
Since ${}_{II}E^2_{1,0}(\cW)={}_{II}E^\infty_{1,0}(\cW)$, we know that
$[x]$ persists nonzero to ${}_{II}E^\infty(\cW)$.  In other words
there exists $\eta\in D_{0,1}(\cW)$ such that $\eta+x$ is a cycle in
$D_1(\cW)$ and represents a nonzero class in $H_1(\cW)$.  By the
paragraph preceding the theorem, we obtain that
$\theta_*(\eta+x)=\phi_{Sh,*}\circ\phi_{\cW,*}\inv(\eta+x)$ is a cycle
in $D_1(Sh)$ and represents a nonzero class in $H_1(Sh)$.

We take $\xi=\theta_*(\eta)$.  Then $\xi+\theta(x)$ is a cycle in
$D_1(Sh)$ and it represents a nonzero class in $H_1(Sh)$.  This proves
assertion (i). Letting $z$ denote this class, we compute $zT$ on the
chain level.  It may help to refer to the following diagrams.  In
$D(\cW)$ we have our cycle
$$
\xymatrix{
\eta & \bullet\\
\bullet & x & \bullet
}
$$
so that $\partial_2\eta=-\partial_1 x$.
Mapping this to 
$D(Sh)$ we have the cycle
$$
\xymatrix{
\xi & \bullet\\
\bullet & \theta(x) & \bullet & 
}
$$
so that 
$\partial_2\xi=-\partial_1 \theta(x)$.

By hypothesis, there exists $u\in D_{11}(Sh)$ and  $y\in D_{20}(Sh)$ such that 
$$
\partial_1 y + \theta(x){\bf s} - \partial_2 u = a\theta(x)
$$ 
on the chain level.
It follows that 
$$
\partial_2(a\xi)  = -a\partial_1\theta(x)  = -\partial_1\theta(x){\bf s} + \partial_1\partial_2 u = 
\partial_2 \xi{\bf s} + \partial_2\partial_1 u.
$$ 
Since the columns of $D$ are exact, there exists $w\in D_{02}(Sh)$ such that
$$
 a\xi= \xi {\bf s}+\partial_1 u + \partial_2 w.
$$
In pictures,
$$
\xymatrix{
w\\
\xi {\bf s} & u\\
\bullet & \theta(x){\bf s} & y & 
}
$$
shows that 
$(\xi+\theta(x)){\bf s}$ is homologous to $a(\xi+\theta(x))$ in
$H_1(Sh)$.  This proves (ii) and completes the proof of the theorem.
\end{proof}

\bibliographystyle{amsalpha_no_mr}
\bibliography{AGM-V}

\providecommand{\bysame}{\leavevmode\hbox to3em{\hrulefill}\thinspace}
\providecommand{\MR}{\relax\ifhmode\unskip\space\fi MR }
\providecommand{\MRhref}[2]{%
  \href{http://www.ams.org/mathscinet-getitem?mr=#1}{#2}
}
\providecommand{\href}[2]{#2}
\begin{thebibliography}{AMRT10}

\bibitem[AGM02]{AGM1}
Avner Ash, Paul~E. Gunnells, and Mark McConnell, \emph{Cohomology of congruence
  subgroups of {${\rm SL}\sb 4(\Bbb Z)$}}, J. Number Theory \textbf{94} (2002),
  no.~1, 181--212.

\bibitem[AGM08]{AGM2}
\bysame, \emph{Cohomology of congruence subgroups of {${\rm SL}(4,\Bbb Z)$}.
  {II}}, J. Number Theory \textbf{128} (2008), no.~8, 2263--2274.

\bibitem[AGM10]{AGM3}
\bysame, \emph{Cohomology of congruence subgroups of {${\rm SL}(4,\Bbb Z)$}.
  {III}}, Math. Comp. \textbf{79} (2010), 1811--1831.

\bibitem[AGM11]{AGM4}
\bysame, \emph{{Torsion in the cohomology of congruence subgroups of $SL(4,\Z)$
  and Galois representations}}, J. Alg., to appear., 2011.

\bibitem[AMRT10]{AMRT}
Avner Ash, David Mumford, Michael Rapoport, and Yung-Sheng Tai, \emph{Smooth
  compactifications of locally symmetric varieties}, second ed., Cambridge
  Mathematical Library, Cambridge University Press, Cambridge, 2010, With the
  collaboration of Peter Scholze.

\bibitem[AR79]{AR}
Avner Ash and Lee Rudolph, \emph{The modular symbol and continued fractions in
  higher dimensions}, Invent. Math. \textbf{55} (1979), no.~3, 241--250.

\bibitem[Ash94]{A2}
Avner Ash, \emph{Unstable cohomology of {${\rm SL}(n,\mathcal O)$}}, J. Algebra
  \textbf{167} (1994), no.~2, 330--342.

\bibitem[Bro94]{B}
Kenneth~S. Brown, \emph{Cohomology of groups}, Graduate Texts in Mathematics,
  vol.~87, Springer-Verlag, New York, 1994, Corrected reprint of the 1982
  original.

\bibitem[BS73]{B-S}
Armand Borel and Jean-Pierre Serre, \emph{Corners and arithmetic groups}, Comm.
  Math. Helv. \textbf{48} (1973), 436--491.

\bibitem[Gun00]{experimental}
Paul~E. Gunnells, \emph{Computing {H}ecke eigenvalues below the cohomological
  dimension}, Experiment. Math. \textbf{9} (2000), no.~3, 351--367.

\bibitem[LS76]{LS}
Ronnie Lee and R.~H. Szczarba, \emph{On the homology and cohomology of
  congruence subgroups}, Invent. Math. \textbf{33} (1976), no.~1, 15--53.

\bibitem[McC91]{M}
Mark McConnell, \emph{Classical projective geometry and arithmetic groups},
  Math. Ann. \textbf{290} (1991), no.~3, 441--462.

\bibitem[Spa81]{spanier}
Edwin~H. Spanier, \emph{Algebraic topology}, Springer-Verlag, New York, 1981,
  Corrected reprint.

\bibitem[Ste07]{G}
William Stein, \emph{Modular forms, a computational approach}, Graduate Studies
  in Mathematics, vol.~79, American Mathematical Society, Providence, RI, 2007,
  With an appendix by Paul E. Gunnells.

\end{thebibliography}

\end{document}